\documentclass[bezier,amstex, oneside,reqno]{amsart}
\usepackage{amsmath, amssymb, amsthm, verbatim, euscript}
\usepackage[dvips]{graphicx}
\usepackage{epsfig}
\usepackage{color}

\def\eg{\emph{e.g., }}

\def\TTT{\mathbb T^3}

\def\eg{\emph{e.g., }}

\def\Wwu{\mathcal W^{wu}}
\def\Wu{\mathcal W^{u}}
\def\Wsu{\mathcal W^{su}}
\def\tWwuL{{\widetilde{\mathcal W}}^{wu}_L}

\def\tWwuf{{\widetilde{\mathcal W}}^{wu}_f}
\def\tWsuf{{\widetilde{\mathcal W}}^{su}_f}
\def\tWwu{{\widetilde{\mathcal W}}^{wu}}
\def\tWsu{{\widetilde{\mathcal W}}^{su}}
\def\U{\mathcal U}
\def\V{\mathcal V}
\def\W{\mathcal W}
\def\P{\EuScript P}
\def\Pu{{{\EuScript P}^{u}}}
\def\Pwu{{\EuScript P^{wu}}}

\def\Dwu{D_f^{wu}}
\def\Dsu{D_f^{su}}

\def\dsu{d^{su}}
\def\tdsu{\tilde d^{su}}
\def\dwu{d^{wu}}
\def\lsu{\lambda^{su}}
\def\Diff{\mbox{\it{Diff}}_m^{\;r}(\TTT)}

\newtheorem*{theoremA}{Theorem A}

\newtheorem{q}{Question}
\newtheorem{theorem}{Theorem}
\newtheorem{lemma}[theorem]{Lemma}
\newtheorem{prop}[theorem]{Proposition}

\theoremstyle{remark}
\newtheorem*{remark}{Remark}

\theoremstyle{remark}
\newtheorem{definition}{Definition}

\begin{document}
\author{ Andrey Gogolev}
\title[Pathological foliations]{How typical are pathological foliations in partially hyperbolic dynamics: an example}
\begin{abstract}
We show that for a large space of volume preserving partially
hyperbolic diffeomorphisms of the 3-torus with non-compact central
leaves the central foliation generically is non-absolutely
continuous.
\end{abstract}
 \maketitle

\section{Introduction}

Let $M$ be a smooth Riemannian manifold. In this paper we will
consider continuous foliations of $M$ with smooth leaves. A {\it
continuous foliation $\W$ with smooth leaves $\W(x)$, $x\in M$,} is
a foliation given by continuous charts whose leaves are smoothly
immersed and whose tangent distribution $T\W$ is continuous on $M$.
Riemannian metric induces volume $m$ on $M$ as well as volume on the
leaves of $\W$. Following Shub and Wilkinson~\cite{SW} we call such
foliation $\W$ {\it pathological} if there is a full volume set on
$M$ that meets every leaf of the foliation on a set of leaf-volume
zero. According to Fubini Theorem, smooth foliations cannot be
pathological, but continuous foliations might happen to be
pathological. This phenomenon naturally appears for central
foliations of partially hyperbolic diffeomorphisms and is also known
as ``Fubini's nightmare." A diffeomorphism $f$ is called partially
hyperbolic if the tangent bundle $TM$ splits into a $Df$-invariant
direct sum of an exponentially contracting stable bundle, an
exponentially expanding unstable bundle and a central bundle of
intermediate growth (precise definitions appear in the next
section).

The first example of a pathological foliation was constructed  by
Katok and it has been circulating in dynamics community since the
eighties. Katok suggested to consider one parameter family $\{A_t,
t\in \mathbb R/\mathbb Z\}$ of area-preserving Anosov
diffeomorphisms $C^1$-close to a hyperbolic automorphism $A$ of the
2-torus. By Hirsch-Pugh-Shub Theorem diffeomorphism
$F(x,t)=(A_t(x),t)$ is partially hyperbolic with uniquely integrable
central distribution. Then, under certain generic conditions (the
metric entropy or periodic eigendata of $A_t$ should vary with $t$)
on path $A_t$, one can show that the central foliation by embedded
circles is pathological. See~\cite{Pes}, Section~7.4,
or~\cite{HassP}, Section~6, for detailed constructions with proofs.

A version of above construction on the square appeared in expository
paper by Milnor~\cite{Milnor}. Milnor remarks that a different
version of the construction, based on tent maps, has also been given
by Yorke.

Shub and Wilkinson~\cite{SW} came across the same phenomenon when
looking for volume preserving non-uniformly hyperbolic systems in
the neighborhood of $F_0\colon (x,t)\mapsto(A_0(x),t)$. They have
showed existence of $C^1$-open set of diffeomorphisms in the
$C^1$-neighborhood of $F_0$ with non-zero central exponent. Then one
can argue that the central foliation is pathological using the
following ``Ma\~n\'e's argument". By Oseledets' Theorem the set of
Lyapunov regular points has full volume. If any central leaf
intersected the set of regular points by a set of positive Lebesgue
measure, then it would increase exponentially in length under the
dynamics. But the lengths of central leaves are uniformly bounded.

Work~\cite{SW} was further generalized by Ruelle~\cite{Ruelle}.
Ruelle and Wilkinson~\cite{RW} also showed that conditional measures
are in fact atomic. Case of higher dimensional central leaves was
considered by Hirayama and Pesin~\cite{HirP}. They showed that
central foliation is not absolutely continuous if it has compact
leaves and the sum of the central exponents is nonzero on a set of
positive measure.

This work is devoted to the study of pathological foliations with
one-dimensional non-compact leaves. Consider a hyperbolic
automorphism $L$ of the 3-torus $\TTT$  with eigenvalues $\nu$,
$\mu$ and $\lambda$ such that $\nu<1<\mu<\lambda$. One can view $L$
as a partially hyperbolic diffeomorphism. It was noted in~\cite{GG}
and independently in~\cite{SX} that for a small $C^1$-open set in
the neighborhood of $L$ ``Ma\~n\'e's argument" can be applied to
show that corresponding central foliations are pathological. In this
paper we apply a completely different approach to show that {\it
there is an open and dense set $\U$ of a large $C^1$-neighborhood of
$L$ in the space of volume preserving partially hyperbolic
diffeomorphisms such that all diffeomorphisms from $\U$ have
pathological central foliations.} This result confirms a conjecture
from~\cite{HirP}.

{\bf Acknowledgement.} The author is grateful to Boris Hasselblatt
and Anatole Katok for listening to the preliminary version of the
proof of the result. The author would like to thank the referees for
useful feedback.

\section{Preliminaries}

Here we introduce all necessary notions and some standard tools that
we need for precise formulation of the result and the proof. The
reader may consult~\cite{Pes} for an introduction on partially
hyperbolic dynamics.
\begin{definition} A diffeomorphism $f$ is called {\it Anosov} if there
exists a $Df$-invariant splitting of the tangent bundle
$TM=E^s_f\oplus E^u_f$ and constants $\lambda\in(0,1)$ and $C>0$
such that for $n>0$
$$
\|Df^nv\|\le C\lambda^n\|v\|,\;v\in
E^s\;\;\;\mbox{and}\;\;\;\|Df^{-n}v\|\le C\lambda^{n}\|v\|,\;v\in
E^u.
$$
\end{definition}
\begin{definition}
A diffeomorphism $f$ is called {\it partially hyperbolic} if there
exists a $Df$-invariant splitting of the tangent bundle
$TM=E^s_f\oplus E^c_f\oplus E^u_f$ and positive constants
$\nu_-<\nu_+<\mu_-<\mu_+<\lambda_-<\lambda_+$, $\nu_+<1<\lambda_-$,
and $C>0$ such that for $n>0$
 \begin{multline*}
{\frac1C\nu_-^n\|v\|\le\|D(f^n)(x)v\|\le C\nu_{+}^n\|v\|,\;\;\; v\in E^s_f(x),}\\
\shoveleft{\frac1C\mu_{-}^n\|v\|\le\|D(f^n)(x)v\|\le C\mu_{+}^n\|v\|,\;\;\; v\in E_{f}^c(x),}\\
\shoveleft{\frac1C\lambda_{-}^n\|v\|\le\|D(f^n)(x)v\|\le
C\lambda_{+}^n\|v\|,\;\;\; v\in E_{f}^u(x).} \hfill
\end{multline*}
 \label{def_phd_usual}
\end{definition}
The following definition is equivalent to the above one. We will
switch between the definitions when convenient.
\begin{definition}
A diffeomorphism $f$ is called {\it partially hyperbolic} if there
exists a Riemannian metric on $M$, a $Df$-invariant splitting of the
tangent bundle $TM=E^s_f\oplus E^c_f\oplus E^u_f$ and positive
constants $\nu_-<\nu_+<\mu_-<\mu_+<\lambda_-<\lambda_+$,
$\nu_+<1<\lambda_-$, such that
 \begin{multline*}
{\nu_-\|v\|\le\|Df(x)v\|\le \nu_{+}\|v\|,\;\;\; v\in E^s_f(x),}\\
\shoveleft{\mu_{-}\|v\|\le\|Df(x)v\|\le \mu_{+}\|v\|,\;\;\; v\in E_{f}^c(x),}\\
\shoveleft{\lambda_{-}\|v\|\le\|Df(x)v\|\le \lambda_{+}\|v\|,\;\;\;
v\in E_{f}^u(x).} \hfill
\end{multline*}
\label{def_phd}
\end{definition}
The distributions $E^s_f$, $E^c_f$ and $E^u_f$ are continuous.
Moreover, distributions $E^s_f$ and $E_f^u$ integrate uniquely to
foliations $\W_f^s$ and $\W_f^u$. When it does not lead to a
confusion we drop dependence on the diffeomorphism. By
$m_{\W^\sigma(\cdot)}$ or $m_\sigma$ we denote induced Riemannian
volume on the leaves of $\W^\sigma$, $\sigma=s, c, u$. Induced
volume on other submanifolds such as transversals to a foliation
will be denoted analogously with appropriate subscript.

We write $d$ for the distance induced by the Riemannian metric and
$d^\sigma(\cdot,\cdot)$ for the distance induced by the restriction
of the Riemannian metric to $T\W^\sigma$. If expanding foliation
$\W^u$ is one-dimensional then it is convenient to work with the
pseudo-distance $\tilde d^u(\cdot,\cdot)$ that is very well adapted
to the dynamics. Let
$$
D_f^u(x)=\|Df(x)\big|_{E_f^{u}(x)}\|
$$
and
$$
\rho_x(y)=\prod_{n\ge 1}\frac{D_f^u(f^{-n}(x))}{D_f^u(f^{-n}(y))}.
$$
This infinite product converges and gives  a continuous positive
density $\rho_x(\cdot)$ on the leaf $\W^u(x)$. Define {\it
pseudo-distance $\tilde d^u$} by integrating density $\rho_x(\cdot)$
$$
\tilde d^u(x,y)=\int_x^y\rho_x(z)dm_{\W^u(x)}(z).
$$
Obviously, pseudo-distance is not even symmetric, still it is useful
for computations as it satisfies the formula
$$
\tilde d^u(f(x),f(y))=D_f^u(x)\tilde d^u(x,y)
$$
verified by the following simple computation
\begin{multline*}
\tilde
d^u(f(x),f(y))=\int_{f(x)}^{f(y)}\rho_{f(x)}(z)dm_{\W^u(f(x))}(z)\\
=\int_x^y\rho_{f(x)}(f(z))D_f^u(z)dm_{\W^u(x)}(z)\\
=\int_x^y\frac{D_f^u(x)}{D_f^u(z)}\rho_{x}(z)D_f^u(z)dm_{\W^u(x)}(z)
=D_f^u(x)\tilde d^u(x,y).
\end{multline*}

A compact domain inside a leaf $\W^\sigma(x)$ of a foliation
$\W^\sigma$ will be called {\it plaque} and will be denoted by
$\P^\sigma$. We shall also write $\P^\sigma(x)$ when we need to
indicate dependence on the point.

Given a transversal $T$ to $\W$, consider a compact domain $X$ which
is a union of plaques of $\W$, that is, $X=\cup_{x\in T}\P(x)$. Then
by Rokhlin's Theorem there exists a unique system of conditional
measures $\mu_x$, $x\in T$, such that for any continuous function
$\varphi$ on $X$
$$
\int_X\varphi dm_X=\int_T\int_{\P(x)}\varphi d\mu_xd\hat m,
$$
where $\hat m$ is projection of $m_X$ to $T$.
\begin{definition}
\label{def_ac} Foliation $\W$ is called {\it absolutely continuous}
with respect to the volume $m$ if for any $T$ and $X$ as above the
conditional measures $\mu_x$ have $L^1$ densities with respect to
the volume $m_{\P(x)}$ for $\hat m$ a.~e. $x$.
\end{definition}
Now consider a compact domain $X$ as above and two transversal $T_1$
and $T_2$ so that $X=\cup_{x\in T_1}\P(x)=\cup_{x\in T_2}\P(x)$ with
the same system of plaques. Then the holonomy map $p\colon T_1\to
T_2$ along $\W$ is a homeomorphism.
\begin{definition}
Foliation $\W$ is called {\it transversally absolutely continuous}
if any holonomy map $p$ as above is absolutely continuous, that is,
$p_*m_{T_1}$ is absolutely continuous with respect to $m_{T_2}$.
\end{definition}

Transverse absolute continuity is a stronger property than absolute
continuity. Stable and unstable foliations of Anosov and partially
hyperbolic diffeomorphisms are known to be transversally absolutely
continuous.

\section{Formulation of the result}
Let $L$ be a hyperbolic automorphism of 3-torus $\TTT$ with positive
real eigenvalues  $\nu$, $\mu$ and $\lambda$, $\nu<1<\mu<\lambda$.
Observe that $L$ can be viewed as a partially hyperbolic
diffeomorphism with $L$-invariant splitting $T\TTT=E_L^s\oplus
E_L^{wu}\oplus E_L^{su}$, where ``wu" and ``su" stand for ``weak
unstable" and ``strong unstable".

Consider the space $\Diff$ of $C^r$, $r\ge 2$, diffeomorphisms of
$\TTT$ that preserve volume $m$. Let $\U\subset\Diff$ be the set of
Anosov diffeomorphisms conjugate to $L$ via a conjugacy homotopic to
identity and also partially hyperbolic. It is known that $\U$ is
$C^1$-open (\eg see~\cite{Pes}, Theorem~3.6). Given $f$ in $\U$
denote by $E_f^s\oplus E_f^{wu}\oplus E_f^{su}$ corresponding
$f$-invariant splitting. According to~\cite{BBI} distributions
$E_f^{wu}$, $E_f^s\oplus E_f^{wu}$ and $E_f^u=E_f^{wu}\oplus
E_f^{su}$ integrate uniquely to invariant foliations $\Wwu$,
$\mathcal W^{s\oplus wu}$ and $\Wu$. It is known that $\mathcal W^s$
and $\Wu$ are $C^1$ and $\Wsu$ is $C^1$ when restricted to the
leaves of $\Wu$ (see, \eg~\cite{Hass, PSW}). We shall need the
following statement that shows that the structure of weak unstable
foliation is essentially linear.

\begin{prop}
\label{prop_slow_to_slow} Let $f\in\U$ and let $h_f$ be the
conjugacy to the linear automorphism---$h_f\circ f=L\circ h_f$. Then
$h_f(\Wwu_f)=\Wwu_L$.
\end{prop}

The proof will be given in the appendix.

\begin{theoremA} There is a $C^1$-open and $C^{r}$-dense set
$\V\subset\U$ such that $f\in\V$ if and only if the central
foliation $\Wwu$ is non-absolutely continuous with respect to the
volume $m$.
\end{theoremA}
\begin{remark}
Since we know that $\Wu$ is $C^1$ the latter is equivalent to $\Wwu$
being non-absolutely continuous on almost every plaque of $\Wu$ with
respect to the induced volume on the plaque.
\end{remark}

Now we describe set $\V$. Given $f\in\U$ and given a periodic point
$x$ of period $p$  let
$$
\lsu(x)=\|Df^p(x)\big|_{E_f^{su}(x)}\|^{1/p}.
$$
Then set $\V$ can be characterized as follows.
$$
\V=\{f\in\U:\mbox{there exist periodic points\;\;} x \mbox{\;and\;}
y \mbox{\;with\;} \lsu(x)\neq\lsu(y)\}.
$$
\begin{prop}
\label{prop1} $$ \U\backslash\V=\{f\in\U:\mbox{\textup{for any
periodic point}\;\;} x \;\;\lsu(x)=\lambda\}.
 $$
\end{prop}
We defer the proof to the appendix.

\section{Related questions}

Our result does not give any information about the structure of
singular conditional measures.
\begin{q}
Given $f\in\V$, what can one say about singular conditional measures
on $\W^{wu}$? Are they atomic? What can be said about Hausdorff
dimension of conditional measures?
\end{q}

It seems that our method can be generalized for analysis of central
foliation of partially hyperbolic diffeomorphisms in a $C^1$
neighborhood of $F_0\colon (x,t)\mapsto(A_0(x),t)$.
\begin{q} Is it true that a generic perturbation of $F_0$ has
non-absolutely continuous central foliation? Can one give explicit
conditions in terms of stable and unstable Lyapunov exponents of
periodic central leaves for non-absolute continuity?
\end{q}

It would be interesting to generalize Theorem~A to the higher
dimensional setting. Namely, let $L$ be an Anosov automorphism that
leaves invariant a partially hyperbolic splitting $E_L^s\oplus
E_L^{wu}\oplus E_L^{su}$, where $E_L^{wu}\oplus E_L^{su}$ is the
splitting of the unstable bundle into weak and strong unstable
subbundles. Let $n_1$, $n_2$ and $n_3$ be the dimensions of $E_L^s$,
$E_L^{wu}$ and $E_L^{su}$ respectively. Let $\U$ be a small $C^1$
neighborhood of $L$ in the space of volume preserving
diffeomorphisms.
\begin{q}
Is it possible to describe the set
$$
\{f\in\U: \Wwu\; \mbox{is not absolutely continuous} \}
$$
in terms of strong unstable spectra at periodic points in higher
dimensional setting?
\end{q}

It will become clear from the discussion in the next section that
the value of $n_1$ is not important. Also it seems likely that our
approach works in the case when $n_2>1$ and $n_3=1$, and gives a
result analogous to Theorem~A (the author does not claim to have
done this).

The picture gets much more complicated when $n_3>1$. It is possible
that the major link in our argument
$$
(\Wwu\;\mbox{is Lipschitz inside}\;\;
\Wu)\Leftrightarrow(\Wwu\;\mbox{is absolutely continuous
inside}\;\;\Wu)
$$
is no longer valid in this setting. However it is not immediately
clear how to construct a counerexample.

\section{Outline of the proof}

Clearly $\V$ is $C^1$-open. Given a diffeomorphism
$f\in\U\backslash\V$ we can compose it with a special diffeomorphism
$h$ that is $C^r$-close to identity and equal to identity outside a
small neighborhood of a fixed point so that strong unstable
eigenvalue of $f$ and $h\circ f$ at the fixed point are different.
This gives that $\V$ is $C^r$-dense.

To show that weak unstable foliations of diffeomorphisms from $\V$
are non-absolutely continuous we start with some simple
observations. First, notice that due to ergodicity conditional
measures cannot have absolutely continuous and singular components
simultaneously. Next, it follows from the absolute continuity of
$\Wu$ and the uniqueness of the system of conditional measures of
$m$ that the conditional measures of $m$ on the leaves of $\Wwu$ are
equivalent to the conditional measures of the induced volume on the
leaves of $\Wu$. Therefore we only need to look at two dimensional
plaques of $\Wu$ foliated by plaques of $\Wwu$. It turns out that
absolute continuity of $\Wwu$ inside the leaves of $\Wu$ is
equivalent to $\Wwu$ being Lipschitz inside $\Wu$. Lipschitz
property, in turn, can be related to the periodic eigenvalue data
along $\Wsu$.

Pick a plaque $\Pu$ of $\Wu$ and let $T_1\subset\Pu$ and
$T_2\subset\Pu$ be two smooth compact transversals to $\Wwu$ with
holonomy map $p\colon T_1\to T_2$. If $p$ is Lipschitz for any
choice of plaque and transversals then we say that $\Wwu$ is {\it
Lipshitz inside} $\Wu$.

Theorem~A follows from the following lemmas
\begin{lemma}
\label{lemma_Lip_pd} Foliation $\Wwu$ is Lipschitz inside $\Wu$ if
and only if $f\in\U\backslash\V$.
\end{lemma}
\begin{lemma}
\label{lemma_Lip_AC} Foliation $\Wwu$ is Lipschitz inside $\Wu$ if
and only if $\Wwu$ is absolutely continuous inside $\Wu$.
\end{lemma}

\section{Proofs}
Let us begin with a useful observation. If one needs to show that
$\Wwu$ is Lipschitz in a plaque $\Pu$ then it is sufficient to check
Lipschitz property of the holonomy map for pairs of transversals
that belong to a smooth family that foliates $\Pu$, \eg plaques of
$\Wsu$. Therefore we can always assume that the transversals are
plaques of $\Wsu$.

\begin{proof}[Proof of Lemma~\ref{lemma_Lip_pd}]
First assume that $f\in\U\backslash\V$. Then Lipschitz property of
$\Wwu$ is shown below by a standard argument that uses Livshits
Theorem.

Let $T_1$ and $T_2$ be two local leaves of $\Wsu$ in a plaque $\Pu$
and let $p\colon T_1\to T_2$ be the holonomy along $\Wwu$.

For $x,y\in T_1$ with $d^{su}(x,y)\ge 1$ the Lipschitz property
\begin{equation}
\label{Lip_prop_far} d^{su}(p(x),p(y))\le C d^{su}(x,y),\;\;
d^{su}(x,y)\ge 1,
\end{equation}
follows from compactness for uniformly bounded plaques $\Pu$. It
might happen that $f^n(x)$ and $f^n(p(x))$ are far from each other
on $\Wwu(f^n(x))$. Hence we need~(\ref{Lip_prop_far}) with uniform
$C$ not only on plaques $\Pu$ of bounded size but also on plaques
that are long in the weak unstable direction. In this
case~(\ref{Lip_prop_far}) cannot be guaranteed solely by compactness
but easily follows from Proposition~\ref{prop_slow_to_slow}.

For $x$ and $y$ close to each other we may use $\tilde d^{su}$
rather than $d^{su}$ since $\tilde d^{su}$ is given by an integral
of a continuous density. Then
$$
\frac{\tilde d^{su}(p(x),p(y))}{\tilde d^{su}(x,y)}=
\prod_{i=0}^{n-1}\frac{\Dsu(f^i(p(x)))}{\Dsu(f^i(x))}\cdot\frac{\tilde
d^{su}(f^n(p(x)),f^n(p(y)))}{\tilde d^{su}(f^n(x),f^n(y))},
$$
where $n$ is chosen so that $d^{su}(f^{n-1}(x),f^{n-1}(y))<1\le
d^{su}(f^{n}(x),f^{n}(y))$. The Lipschitz estimate follows since
according to the Livshits Theorem $\Dsu$ is cohomologous to
$\lambda$ and therefore the product term equals to
$F(f^n(x))F(f^n(p(x))(F(x)F(p(x)))^{-1}$ for some positive
continuous transfer function $F$.

Now let us take $f$ from $\V$. Specification implies that the
closure of the set $\{\lsu(x): x\;\mbox{periodic}\}$ is an interval
$[\lsu_-, \lsu_+]$. By applying Anosov Closing Lemma  it is possible
to change the Riemannian metric so that the constants $\lambda_-$
and $\lambda_+$ from Definition~\ref{def_phd} are equal to
$\lsu_-/(1+\delta)$ and $\lsu_+(1+\delta)$ correspondingly. Here
$\delta$ is an arbitrarily small number.

Next we choose periodic points $a$ and $b$ such that
$$
\max\left\{\frac{\lsu_+}{\lsu(a)},
\frac{\lsu(b)}{\lsu_-}\right\}\le1+\delta \;\;\;  \mbox{and}\;\;\;
\frac{\lsu(b)}{\lsu(a)}\le\frac{1}{(1+\delta)^{2/\gamma}}.
$$
This is possible if $\delta$ is small enough. From now on $\delta$
will be fixed. Constant $\gamma$ does not depend on our choice of
$a$ and $b$ and hence $\delta$. It will be introduced later.

Denote by $n_0$ the least common period of $a$ and $b$. Take $\tilde
a\in\Wsu(a)$ such that $\dsu(a,\tilde a)=1$. If one considers an arc
of a leaf of $W_L^{wu}$ of length $D$ then it is easy to see that
this arc is $const/\sqrt D$-dense in $\TTT$. Since conjugacy $h_f$
between $f$ and $L$ is H\"older continuous,
Proposition~\ref{prop_slow_to_slow} implies that an arc of $\Wwu(a)$
of length $D$ is $C_1/D^\alpha$-dense in $\TTT$, $\alpha>0$. It
follows that there exists a point $c\in\Wwu(a)$ such that
$\dwu(a,c)\le D$, $d(c,b)\le C_1/D^\alpha$ and $W^s(b)$ intersects
the arc of strong unstable leaf $\Wsu(c)$ that connects $c$ and
$\tilde c=\Wsu(c)\cap\Wwu(\tilde a)$ at point $\tilde b$ as shown on
the Figure~\ref{fig_hold}.


\begin{figure}[htbp]
\begin{center}

\begin{picture}(0,0)%
\includegraphics{Hold_hol.pstex}%
\end{picture}%
\setlength{\unitlength}{3947sp}%
\begingroup\makeatletter\ifx\SetFigFont\undefined%
\gdef\SetFigFont#1#2#3#4#5{%
  \reset@font\fontsize{#1}{#2pt}%
  \fontfamily{#3}\fontseries{#4}\fontshape{#5}%
  \selectfont}%
\fi\endgroup%
\begin{picture}(5574,3333)(739,218)
\put(2713,2039){\makebox(0,0)[lb]{\smash{{\SetFigFont{8}{9.6}{\rmdefault}{\mddefault}{\updefault}{\color[rgb]{0,0,0}$\mathcal W^{wu}(a)$}%
}}}}
\put(1617,281){\makebox(0,0)[lb]{\smash{{\SetFigFont{8}{9.6}{\rmdefault}{\mddefault}{\updefault}{\color[rgb]{0,0,0}$\tilde a$}%
}}}}
\put(2177,434){\makebox(0,0)[lb]{\smash{{\SetFigFont{8}{9.6}{\rmdefault}{\mddefault}{\updefault}{\color[rgb]{0,0,0}$\mathcal W^{su}(a)$}%
}}}}
\put(3450,1452){\makebox(0,0)[lb]{\smash{{\SetFigFont{8}{9.6}{\rmdefault}{\mddefault}{\updefault}{\color[rgb]{0,0,0}$\mathcal W^{wu}(\tilde a)$}%
}}}}
\put(5843,3183){\makebox(0,0)[lb]{\smash{{\SetFigFont{8}{9.6}{\rmdefault}{\mddefault}{\updefault}{\color[rgb]{0,0,0}$\tilde c$}%
}}}}
\put(5181,2419){\makebox(0,0)[lb]{\smash{{\SetFigFont{8}{9.6}{\rmdefault}{\mddefault}{\updefault}{\color[rgb]{0,0,0}$W^s(b)$}%
}}}}
\put(4681,3437){\makebox(0,0)[lb]{\smash{{\SetFigFont{8}{9.6}{\rmdefault}{\mddefault}{\updefault}{\color[rgb]{0,0,0}$c$}%
}}}}
\put(5079,3407){\makebox(0,0)[lb]{\smash{{\SetFigFont{8}{9.6}{\rmdefault}{\mddefault}{\updefault}{\color[rgb]{0,0,0}$\tilde b$}%
}}}}
\put(4799,2877){\makebox(0,0)[lb]{\smash{{\SetFigFont{8}{9.6}{\rmdefault}{\mddefault}{\updefault}{\color[rgb]{0,0,0}$b$}%
}}}}
\put(949,535){\makebox(0,0)[lb]{\smash{{\SetFigFont{8}{9.6}{\rmdefault}{\mddefault}{\updefault}{\color[rgb]{0,0,0}$a$}%
}}}}
\end{picture}%

\end{center}
\caption{}\label{fig_hold}
\end{figure}


Take $N$ such that $\dwu(a,f^{-n_0N}(c))\le
1<\dwu(a,f^{-n_0(N-1)}(c))$. Now our goal is to show that the ratio
$$
\frac{\tdsu(a,f^{-n_0N}(\tilde
a))}{\tdsu(f^{-n_0N}(c),f^{-n_0N}(\tilde c))}
$$
can be arbitrarily small which would imply that $\Wsu$ is not
Lipschitz. Note that we cannot take a smaller $N$ since
$f^{-1}$-orbit of $c$ has to come to a local plaque about $a$.
\begin{remark} We use $\tdsu$ for convenience. Somewhat messier
estimates go through if one uses $\dsu$ directly.
\end{remark}

 To estimate the denominator we split the
orbit $\{c, f^{-1}(c),\ldots f^{-n_0N}(c)\}$ into two segments of
lengths $N_1$ and $N_2$, $N_1+N_2=n_0N$. Choose $N_1$ so that
$d(f^{-N_1}(b),f^{-N_1}(c))$ is still small enough to provide the
estimate on the strong unstable derivative
$$
\Dsu(f^{-i}(c))\le(1+\delta)\lsu(b), \;i=1,\ldots\; N_1+1.
$$
The remaining derivatives will be estimated boldly
$$
\Dsu(\cdot)\le\lambda_+.
$$
Since $b$ and $\tilde b$ are exponentially close --- $d^s(b,\tilde
b)\le C_1/D^\alpha\le const\cdot\mu_-^{-n_0N}$ --- we see that there
exists $\beta=\beta(\alpha,\nu_-,\mu_-)$ which is independent of $N$
such that $N_1>\beta N_2$.

Proposition~\ref{prop_slow_to_slow} implies that the ratio
$\tdsu(a,\tilde a)/\tdsu(c,\tilde c)$ is bounded independently of
$D$ (and $N$) by a constant $C_2$. We are ready to proceed with the
main estimate.
\begin{multline*}
\frac{\tdsu(a,f^{-n_0N}(\tilde
a))}{\tdsu(f^{-n_0N}(c),f^{-n_0N}(\tilde c))}=
\prod_{i=1}^{n_0N+1}\frac{\Dsu(f^i(c))}{(\lsu)^{n_0N}}\cdot\frac{\tdsu(a,\tilde
a)}{\tdsu(c,\tilde c)}\\
\le(\lsu(a))^{-n_0N}(1+\delta)^{N_1}(\lsu(b))^{N_1}(1+\delta)^{N_2}(\lsu_+)^{N_2}C_2\\
\le(1+\delta)^{N_1+N_2}\left(\frac{\lsu(b)}{\lsu(a)}\right)^{N_1}\left(\frac{\lsu_+}{\lsu(a)}\right)^{N_2}C_2\\
\le(1+\delta)^{N_1+2N_2}\left(\frac{\lsu(b)}{\lsu(a)}\right)^{\gamma(N_1+2N_2)}C_2
\le\left(\frac{1}{1+\delta}\right)^{n_0N+N_2}C_2,
\end{multline*}
where $\gamma=\beta/\beta+2$ so that $N_1\ge\gamma(N_1+2N_2)$. The
last expression goes to zero as $D\to\infty, N\to\infty$. Thus
$\Wwu$ is not Lipschitz.
\end{proof}

\begin{proof}[Proof of Lemma~\ref{lemma_Lip_AC}] Obviously $\Wwu$ being
Lipschitz implies transverse absolute continuity property and hence
absolute continuity. We have to establish the other implication.

Assume that $\Wwu$ is absolutely continuous in the sense of
Definition~\ref{def_ac}. A priori, conditional densities are only
$L^1$-functions. Our goal is to show that the densities are
continuous. Moreover, for $m$ almost every $x$ the density
$\rho_x(y)$ on a plaque $\Pwu$ satisfies the equation
\begin{equation}
\label{product_formula}
\frac{\rho_x(y)}{\rho_x(x)}=\prod_{n\ge1}\frac{\Dwu(f^{-n}(x))}{\Dwu(f^{-n}(y))},
\end{equation}
where $\Dwu(z)=\|Df\big|_{E_f^{wu}(z)}(z)\|$. The expression on the
right hand side of the formula is a positive continuous function in
$y$.

Consider a full volume set where positive ergodic averages coincide
for all continuous functions. By absolute continuity this set should
intersect a plaque $\Pwu$ by a positive leaf-volume $m_{wu}$ set
$Y$. Denote by $m_Y$ restriction of $m_{wu}$ to $Y$. For any $y\in
Y$ consider measures
$$
\Delta_n(y)=\frac1n\sum_{i=0}^{n-1}\delta_{f^i(y)},\;
\mu_n=\int_Y\Delta_n(y)dm_{wu}(y).
$$
Sequences $\{\Delta_n(y)\}$, $y\in Y$, converge weakly to $m$. Hence
$\mu_n$ converges to $m$ as well. Notice that
$$
\mu_n=\frac1n\sum_{i=0}^{n-1}\int_Y\delta_{f^i(y)}dm_{wu}(y)=\frac1n\sum_{i=0}^{n-1}(f^i)_*(m_Y).
$$
In case when $Y$ is a plaque of $\Wwu$ the latter expression is
known to converge to a measure with absolutely continuous
conditional densities on $\Wwu$ that
satisfy~(\ref{product_formula}). This was established in~\cite{PS}
in the context of $u$-Gibbs measures, however the proof works
equally well for any uniformly expanding foliation such as $\Wwu$.
For arbitrary measurable $Y$ same conclusion holds. One needs to use
Lebesgue density argument to reduce the problem to the case when $Y$
is a finite union of plaques.

\begin{remark}
The argument presented above can also be found in~\cite{BDV},
Section~11.2.2, in the context of $u$-Gibbs measures.
\end{remark}

Take an $m$-typical plaque $\Pu$ whose boundaries are leaves of
$\Wwu$ and transversals $T_1$ and $T_2$ as shown on the
Figure~\ref{fig_lip}. Then plaque $\Pu$ is foliated by the plaques
$\Pwu(x)$, $x\in T_1$. As usual, denote by $p\colon T_1\to T_2$ the
holonomy map. Lipschitz property of $p$ will be established  by
comparing volumes of small rectangles $R_1$ and $R_2$ built on
corresponding segments of $T_1$ and $T_2$.

\begin{figure}[htbp]

\begin{center}
\begin{picture}(0,0)%
\includegraphics{Lip_hol.pstex}%
\end{picture}%
\setlength{\unitlength}{3947sp}%
\begingroup\makeatletter\ifx\SetFigFont\undefined%
\gdef\SetFigFont#1#2#3#4#5{%
  \reset@font\fontsize{#1}{#2pt}%
  \fontfamily{#3}\fontseries{#4}\fontshape{#5}%
  \selectfont}%
\fi\endgroup%
\begin{picture}(3416,2419)(289,-1859)
\put(2251,119){\makebox(0,0)[lb]{\smash{{\SetFigFont{12}{14.4}{\rmdefault}{\mddefault}{\updefault}{\color[rgb]{0,0,0}$R_2$}%
}}}}
\put(676,-1786){\makebox(0,0)[lb]{\smash{{\SetFigFont{12}{14.4}{\rmdefault}{\mddefault}{\updefault}{\color[rgb]{0,0,0}$T_1$}%
}}}}
\put(751,389){\makebox(0,0)[lb]{\smash{{\SetFigFont{12}{14.4}{\rmdefault}{\mddefault}{\updefault}{\color[rgb]{0,0,0}$T_2$}%
}}}}
\put(601,-811){\makebox(0,0)[lb]{\smash{{\SetFigFont{12}{14.4}{\rmdefault}{\mddefault}{\updefault}{\color[rgb]{0,0,0}$\EuScript P^u$}%
}}}}
\put(1426,389){\makebox(0,0)[lb]{\smash{{\SetFigFont{12}{14.4}{\rmdefault}{\mddefault}{\updefault}{\color[rgb]{0,0,0}$p(x)$}%
}}}}
\put(2101,389){\makebox(0,0)[lb]{\smash{{\SetFigFont{12}{14.4}{\rmdefault}{\mddefault}{\updefault}{\color[rgb]{0,0,0}$p(y)$}%
}}}}
\put(1576,-1753){\makebox(0,0)[lb]{\smash{{\SetFigFont{12}{14.4}{\rmdefault}{\mddefault}{\updefault}{\color[rgb]{0,0,0}$x$}%
}}}}
\put(2401,-1753){\makebox(0,0)[lb]{\smash{{\SetFigFont{12}{14.4}{\rmdefault}{\mddefault}{\updefault}{\color[rgb]{0,0,0}$y$}%
}}}}
\put(2551,-1513){\makebox(0,0)[lb]{\smash{{\SetFigFont{12}{14.4}{\rmdefault}{\mddefault}{\updefault}{\color[rgb]{0,0,0}$R_1$}%
}}}}
\end{picture}%

\end{center}
\caption{}\label{fig_lip}
\end{figure}

Denote by $\mu_\Pu$ the conditional measure on $\Pu$. The
conditional densities $\rho_x(\cdot)$ of $m$ on the plaques
$\Pwu(x), x\in T_1$, are the same as conditional densities with
respect to $\mu_\Pu$.

Fix $x, y\in T_1$ and small $\varepsilon>0$, $\varepsilon\ll
m_{T_1}([x,y])$. Build rectangles $R_1$ and $R_2$ on the segments
$[x,y]$ and $[p(x),p(y)]$ so that
$m_{\Pwu(z)}(R_1\cap\Pwu(z))=m_{\Pwu(z)}(R_2\cap\Pwu(z))=\varepsilon$
for every $z\in[x,y]$. Then
$$
\mu_\Pu(R_i)=\int_{[x,y]}d\hat\mu(z)\int_{\Pwu(z)\cap
R_i}\rho_z(t)dm_{\Pwu(z)}(t),\;\; i=1,2,
$$
where $\hat\mu$ is projection of $\mu_\Pu$ to $T_1$. These formulae
together with uniform continuity of the conditional densities that
is guaranteed by~(\ref{product_formula}) imply that the ratio
$\mu_\Pu(R_1)/\mu_\Pu(R_2)$ is bounded away from zero and infinity
uniformly in $x$ and $y$. Since $\mu_\Pu$ has positive continuous
density with respect to $m_\Pu$ the same conclusion holds for
$m_\Pu(R_1)/m_\Pu(R_2)$ and therefore also for
$m_{T_1}([x,y])/m_{T_2}([p(x),p(y)])$.
\end{proof}

\section{Appendix}
Appendix is devoted to the proofs of
Propositions~\ref{prop_slow_to_slow} and~\ref{prop1}. Both proofs
rely on simple growth arguments and a result of Brin, Burago and
Ivanov. We will work on the universal cover $\mathbb R^3$ and we
will indicate this by using tilde sign for lifted objects. For
example, the lift of foliation $\Wsu_f$ to $\mathbb R^3$ is denoted
by $\tWsuf$.

The result of Brin, Burago and Ivanov~\cite{BBI} says that lifts of
leaves of strong unstable foliation are quasi-isometric. Namely, if
$d$ is the usual distance then
$$
\exists C>0:\;\forall x,y\;\mbox{with}\; y\in{\tilde {\mathcal
W}}^{su}(x),\;\; \dsu(x,y)\le Cd(x,y).
$$

\begin{proof}[Proof of Proposition~\ref{prop_slow_to_slow}]
We argue by contradiction. Assume that $\tilde
h_f(\tWwuf)\neq\tWwuL$. Then we can find points $a$, $b$ and $c$
with the following properties
$$
b\in\tWwuL(a), c\notin\tWwuL(a), h_f^{-1}(c)=\tWwuf(\tilde
h_f^{-1}(a))\cap\tWsuf(\tilde h_f^{-1}(b)).
$$
We iterate automorphism $L$ and look at the asymptotic growth of the
distance between these points. Obviously, distance between images of
$a$ and $b$ grows as $\mu^n$, meanwhile distance between images of
$a$ and $c$, and images of $b$ and $c$ grows as $\lambda^n$.

Since conjugacy $\tilde h_f$ is $C^0$-close to $Id$ we have the same
growth rates for the triple $\tilde h_f^{-1}(a)$, $\tilde
h_f^{-1}(b)$ and $\tilde h_f^{-1}(c)$ as we iterate dynamics $\tilde
f$. Points $\tilde h_f^{-1}(a)$ and $\tilde h_f^{-1}(c)$ lie on the
same weak unstable manifold, therefore, constant $\mu_+$ from the
Definition~\ref{def_phd_usual} is not less than $\lambda$. Then,
obviously, $\lambda_->\lambda$. Since $\tWsuf$ is quasi-isometric
$$d(\tilde f^n(\tilde h_f^{-1}(c)),\tilde f^n(\tilde
h_f^{-1}(b)))\approx\dsu(\tilde f^n(\tilde h_f^{-1}(c)),\tilde
f^n(\tilde h_f^{-1}(b)))\gtrsim\lambda_-^n,\; n\to\infty.
$$
On the other hand, we have already established that the distance
between images of $\tilde h_f^{-1}(c)$ and $\tilde h_f^{-1}(b)$
diverges as $\lambda^n$. This gives us a contradiction.
\end{proof}

\begin{proof}[Proof of Proposition~\ref{prop1}]
We argue by contradiction. Assume that $f\in\U\backslash\V$. Then
for every periodic point $x$, $\lsu(x)=\bar\lambda\neq\lambda$.

First assume that $\bar\lambda<\lambda$. Then constant $\lambda_+$
from Definition~\ref{def_phd_usual} can be taken to be equal to
$\frac12(\lambda+\bar\lambda)$.  Pick points $a$ and $b$,
$b\in\tWsu(a)$. Then
$$
d(\tilde f^n(a),\tilde f^n(b))\le\dsu(\tilde f^n(a),\tilde
f^n(b))\lesssim\lambda_+^n,\; n\to\infty.
$$
By Proposition~\ref{prop_slow_to_slow} $\tilde
h_f(b)\notin\tWwu(\tilde h_f(a)$. Therefore,
$$
d(\tilde L^n(\tilde h_f(a)),\tilde L^n(\tilde
h_f(b)))\gtrsim\lambda^n,\;n\to\infty.
$$
On the other hand,
$$
d(\tilde L^n(\tilde h_f(a)),\tilde L^n(\tilde h_f(b)))= d(\tilde
h_f(\tilde f^n(a)),\tilde h_f(\tilde
f^n(b)))\lesssim\lambda_+^n,\;n\to\infty,
$$
since $\tilde h_f$ is $C^0$-close to $Id$. The last two asymptotic
inequalities contradict to each other.

Now let us assume that $\bar\lambda>\lambda$. In this case we can
take $\lambda_-$ from Definition~\ref{def_phd_usual} to be equal to
$\frac12(\lambda+\bar\lambda)$. Take $a$ and $b$ as before. Since
$\tWsu$ is quasi-isometric
$$
d(\tilde f^n(a),\tilde f^n(b))\gtrsim\dsu(\tilde f^n(a),\tilde
f^n(b))\gtrsim\lambda_-^n,\;n\to\infty.
$$
On the other hand,
$$
d(\tilde f^n(a),\tilde f^n(b))\approx d(\tilde h_f(\tilde
f^n(a)),\tilde h_f(\tilde f^n(b)))=d(\tilde L^n(\tilde
h_f(a)),\tilde L^n(\tilde h_f(b)))\lesssim\lambda^n,\;n\to\infty,
$$
which gives us a contradiction in this case as well.
\end{proof}

\end{document}